\documentclass[suppldata]{interact}

\usepackage{epstopdf}
\usepackage[caption=false]{subfig}

\usepackage{color}

\usepackage[numbers,sort&compress]{natbib}
\bibpunct[, ]{[}{]}{,}{n}{,}{,}
\renewcommand\bibfont{\fontsize{10}{12}\selectfont}
\makeatletter
\def\NAT@def@citea{\def\@citea{\NAT@separator}}
\makeatother

\theoremstyle{plain}
\newtheorem{theorem}{Theorem}[section]
\newtheorem{lemma}[theorem]{Lemma}
\newtheorem{corollary}[theorem]{Corollary}

\theoremstyle{definition}
\newtheorem{definition}[theorem]{Definition}

\theoremstyle{remark}
\newtheorem{remark}{Remark}

\begin{document}


\title{The $C$-Numerical Range for Schatten-Class Operators}

\author{
\name{Gunther Dirr\textsuperscript{a} and Frederik vom Ende\textsuperscript{b}\thanks{CONTACT Gunther Dirr. Email: dirr@mathematik.uni-wuerzburg.de, Frederik vom Ende (corresponding author). Email: frederik.vom-ende@tum.de}}
\affil{\textsuperscript{a}Department of Mathematics, University of W{\"u}rzburg, 97074 W{\"u}rzburg, Germany\\
\textsuperscript{b}Department of Chemistry, TU Munich, 85747 Garching, Germany}
}

\maketitle

\begin{abstract}
We generalize the $C$-numerical range $W_C(T)$ from trace-class to Schatten-class operators, i.e. to
$C\in\mathcal B^p(\mathcal H)$ and $T\in\mathcal B^q(\mathcal H)$ with $1/p + 1/q = 1$, and show that
its closure is always star-shaped with respect to the origin.
For $q \in (1,\infty]$, this is equivalent to saying that the closure of the image of the unitary orbit of 
$T\in\mathcal B^q(\mathcal H)$ under any continous linear functional $L\in(\mathcal B^q(\mathcal H))'$ 
is star-shaped with  respect to the origin. For $q=1$, one has star-shapedness with respect to
$\operatorname{tr}(T)W_e(L)$, where $W_e(L)$ denotes the essential range of $L$.

Moreover, the closure of $W_C(T)$ is convex if $C$ or $T$ is normal with collinear eigenvalues.
If $C$ and $T$ are both normal, then the $C$-spectrum of $T$ is a subset of the $C$-numerical
range, which itself is a subset of the closure of the convex hull of the $C$-spectrum. This closure
coincides with the closure of the $C$-numerical range if, in addition, the eigenvalues of $C$ or $T$
are collinear.

\end{abstract}

\begin{amscode}
47A12, 47B10, 15A60
\end{amscode}

\section{Introduction}
This is a follow-up paper of \cite{dirr_ve}. There, we studied the $C$-numerical range $W_C (T)$
of $T$ generalized to trace-class operators $C$ and bounded operators $T$ acting on some infinite-dimensional 
separable complex Hilbert space $\mathcal H$, i.e.
$$
W_C (T) = \lbrace \operatorname{tr}(CU^\dagger TU)\,|\,U\in\mathcal B(\mathcal H)\text{ unitary}\rbrace\,,
$$
where $\mathcal B(\mathcal H)$ denotes the set of all bounded linear operators on $\mathcal H$. In this setting,
however, symmetry in $C$ and $T$ compared to the matrix case is lost in the sense that by construction 
the mapping $(C,T) \mapsto W_C (T)$ is no longer defined on a symmetric domain. Probably, the most natural
symmetric domain where $\operatorname{tr}(CT)$ is still well-defined is the set $\mathcal B^2(\mathcal H)$ of all
Hilbert-Schmidt operators. Thus a natural question to ask is whether the known results about convexity,
star-shapedness and the $C$-spectrum carry over to Hilbert-Schmidt operators. 

While analyzing this problem, it rapidly becomes evident that one can easily go one step further by considering
operators $C$ and $T$ which belong to conjugate Schatten-classes, as the set $\mathcal B^p(\mathcal H)$ of all 
$p$-Schatten-class operators constitutes a two-sided ideal in the $C^*$-algebra $\mathcal B(\mathcal H)$ for all 
$p\in[1,\infty]$. Starting from the symmetry requirement, our line-of-thought will arrive in a quite natural way at the outlined Schatten-class setting of the $C$-numerical range of $T$.

The paper is organized as follows: After a preliminary section collecting notation and basic results on Schatten-class
operators, we present our main results in Section \ref{sec:results}. We show that the closure of $W_C (T)$
for conjugate Schatten-class operators $C$ and $T$ is always star-shaped with respect to the origin. We reformulate
this result in terms of the image of the unitary orbit of $T\in\mathcal B^q(\mathcal H)$ under any continuous
linear functional $L\in(\mathcal B^q(\mathcal H))'$. Moreover, we prove that the closure of $W_C(T)$ is convex 
if either $C$ or $T$ is normal with collinear eigenvalues. Finally, we introduce the $C$-spectrum of $T$
and derive some inclusion and convexitiy results, which are well known for matrices, under the assumption that
both Schatten-class operators $C$ and $T$ are normal.

\section{Notation and Preliminaries}\label{sec:prelim}
Unless stated otherwise, here and henceforth $\mathcal X$ and $\mathcal Y$ are arbitrary 
infinite-dimensional complex Hilbert spaces while $\mathcal H$ and $\mathcal G$ are reserved
for infinite-dimensional \textit{separable} complex Hilbert spaces (for short i.s.c. Hilbert spaces). 
Moreover, $\mathcal B(\mathcal X,\mathcal Y)$, $\mathcal K(\mathcal X,\mathcal Y)$ and 
$\mathcal B^p(\mathcal X,\mathcal Y)$ denote the set of all bounded, compact and $p$-th Schatten-class operators between $\mathcal X$ and $\mathcal Y$, respectively. 

Scalar products are conjugate linear in the first argument and linear in the second one. For an arbitrary subset $S \subset \mathbb{C} $, the notations $\overline{S}$ and $\operatorname{conv}(S)$
stand for its closure and convex hull, respectively. Finally, given $p,q\in[1,\infty]$, we say $p$ and $q$ are conjugate if $\frac1p+\frac1q=1$.

\subsection{Infinite-dimensional Hilbert Spaces and the Trace Class}
For a comprehensive introduction to infinite-dimensional Hilbert spaces and Schatten-class operators, we refer to, e.g., \cite{berberian1976} and \cite{MeiseVogt}. Here, we recall only some
basic results which will be use frequently throughout this paper. \medskip

Let $(e_i)_{i\in I}$ be any orthonormal basis of $\mathcal X$ and let $x \in \mathcal X$. Then one 
has \textit{Parseval's identity} 
\begin{equation*}
\sum_{i\in I}|\langle e_i,x\rangle|^2 = \Vert x\Vert^2
\end{equation*}
which reduces to \textit{Bessel's inequality} 
\begin{equation*}
\sum_{j\in J}|\langle f_j,x\rangle|^2  \leq  \Vert x\Vert^2
\end{equation*}
if $(f_j)_{j\in J}$ is any orthonormal system in $\mathcal X$ instead of an orthonormal basis.
\begin{lemma}[Schmidt decomposition]\label{thm_1}
For each $C \in \mathcal K(\mathcal X,\mathcal Y)$, there exists a decreasing null sequence 
$(s_n(C))_{n\in\mathbb N}$ in $[0,\infty)$ as well as orthonormal systems $(f_n)_{n\in\mathbb N}$ in $\mathcal X$
and $(g_n)_{n\in\mathbb N}$ in $\mathcal Y$ such that
\begin{align*}
C = \sum_{n=1}^\infty s_n(C)\langle f_n,\cdot\rangle g_n\,,
\end{align*}
where the series converges in the operator norm.
\end{lemma}
As the \emph{singular numbers} $(s_n(C))_{n\in\mathbb N}$ in Lemma \ref{thm_1} are uniquely determined by $C$,
the \emph{$p$-th Schatten-class} $\mathcal B^p(\mathcal X,\mathcal Y)$ is defined by
\begin{align*}
\mathcal B^p(\mathcal X,\mathcal Y)
:= \Big\lbrace C \in\mathcal K(\mathcal X,\mathcal Y)\,\Big|\,\sum\nolimits_{n=1}^\infty s_n(C)^p<\infty\Big\rbrace
\end{align*}
for $p\in [1,\infty)$. The Schatten-$p$-norm
\begin{align*}
\nu_p(C) := \Big(\sum_{n=1}^\infty s_n(C)^p\Big)^{1/p}
\end{align*}
turns $\mathcal B^p(\mathcal X,\mathcal Y)$ into a Banach space. Moreover, for $p=\infty$, we 
identify $\mathcal B^\infty (\mathcal X,\mathcal Y)$ with the set of all compact operators 
$\mathcal K(\mathcal X,\mathcal Y)$ equipped with the norm
\begin{align*}
\nu_\infty(C) := \sup_{n \in \mathbb{N}}s_n(C) = s_1(C)\,.
\end{align*}
Note that $\nu_\infty(C)$ coincides with the ordinary operator norm $\|C\|$. Hence 
$\mathcal B^\infty (\mathcal X,\mathcal Y)$ constitutes a closed subspace of 
$\mathcal B (\mathcal X,\mathcal Y)$ and thus a Banach space, too. The following
results can be found in \cite[Coro.~XI.9.4 \& Lemma XI.9.9]{dunford1963linear}.

\begin{lemma}\label{lemma_10}
\begin{itemize}
\item[(a)] Let $p\in [1,\infty]$. Then, for all $S,T\in\mathcal B(\mathcal X)$ and 
$C\in\mathcal B^p(\mathcal X)$, one has
\begin{align*}
\nu_p(SCT)\leq\Vert S\Vert\nu_p(C)\Vert T\Vert\,.
\end{align*}
\item[(b)] Let $1\leq p\leq q\leq\infty$. Then $\mathcal B^p(\mathcal X,\mathcal Y)\subseteq\mathcal B^q(\mathcal X,\mathcal Y)$ and $\nu_p( C)\geq\nu_q(C)$ for all $C\in\mathcal B^p(\mathcal X,\mathcal Y)$.
\end{itemize}
\end{lemma}
\noindent Note that due to (a), all Schatten-classes $\mathcal B^p(\mathcal X)$ constitute -- just
like the compact operators -- a two-sided ideal in the $C^*$-algebra of all bounded operators
$\mathcal B(\mathcal X)$.\medskip

\begin{lemma}\label{lemma_2b}
Let $T\in\mathcal K(\mathcal X)$ and $(e_k)_{k\in\mathbb N}$ be any orthonormal system in $\mathcal X$. Then\vspace{2pt}
\begin{itemize}
\item[(a)]
$
\displaystyle
\sum\nolimits_{k=1}^n|\langle e_k,Te_k\rangle|
\leq\sum\nolimits_{k=1}^n s_k(T)
$
for all $n\in\mathbb N$ and\vspace{6pt}
\item[(b)]
$\lim_{k \to \infty} \langle e_k,Te_k\rangle = 0\,.$
\end{itemize}
\end{lemma}

\begin{proof}
(a) Consider a Schmidt decomposition $\sum_{m=1}^\infty s_m(T)\langle f_m,\cdot\rangle g_m$ of $T$. Then
\begin{align*}
\sum_{k=1}^n|\langle e_k,Te_k\rangle|\leq \sum_{m=1}^\infty s_m(T)\Big( \underbrace{\sum_{k=1}^n |\langle e_k,f_m\rangle\langle g_m,e_k\rangle|}_{=:\lambda_m} \Big)\,.
\end{align*}
Note that by Cauchy-Schwarz and Bessel's inequality one has 
\begin{align*}
\lambda_m\leq \Big(\sum_{k=1}^n |\langle e_k,f_m\rangle|^2\Big)^{1/2}\Big(\sum_{k=1}^n |\langle g_m,e_k\rangle|^2\Big)^{1/2}\leq 1
\end{align*}
for all $m\in\mathbb N$. On the other hand, Cauchy-Schwarz and Bessel's inequality also imply
\begin{align*}
\sum_{m=1}^\infty \lambda_m&\leq \sum_{k=1}^n \Big(\sum_{m=1}^\infty |\langle e_k,f_m\rangle|^2\Big)^{1/2}\Big(\sum_{m=1}^\infty |\langle g_m,e_k\rangle|^2\Big)^{1/2}\leq \sum_{k=1}^n \|e_k  \|^2=n\,.
\end{align*}
Hence an upper bound of $\sum_{m=1}^\infty s_m(T)\lambda_m$ is given by choosing $\lambda_1=\ldots=\lambda_n=1$
and $\lambda_j=0$ whenever $j>n$, since $s_1(T)\geq s_2(T)\geq\ldots$ by construction. This shows the 
desired inequality. A proof of (b) can be found, e.g., in \cite[Lemma 16.17]{MeiseVogt}.
\end{proof}

Now for any $C\in\mathcal B^1(\mathcal X)$, the trace of $C$ is defined via
\begin{align}\label{eq:trace}
\operatorname{tr}(C):=\sum\nolimits_{i\in I}\langle f_i,Cf_i\rangle\,,
\end{align}
where $(f_i)_{i\in I}$ can be any orthonormal basis of $\mathcal X$. The trace is well-defined as
one can show that the right-hand side of \eqref{eq:trace} does not depend on the choice of 
$(f_i)_{i\in I}$. Important properties are the following, cf. \cite[Lemma XI.9.14]{dunford1963linear}.

\begin{lemma}\label{lemma_nu_hoelder}
Let $C\in\mathcal B^p(\mathcal X)$ and $T\in\mathcal B^q(\mathcal X)$ with $p,q\in [1,\infty]$ conjugate. Then one has $CT,TC\in\mathcal B^1(\mathcal X)$ with
\begin{align}
\operatorname{tr}(CT)&=\operatorname{tr}(TC)\nonumber\\
|\operatorname{tr}(CT)|&\leq \nu_p(C)\nu_q(T)\,.\label{eq:4}
\end{align}
\end{lemma}
\noindent
Note that the space of so called Hilbert-Schmidt operators $\mathcal B^2(\mathcal X)$ turns into a
Hilbert space under the scalar product $\langle C,T\rangle:=\operatorname{tr}(C^\dagger T)$
\cite[Prop.~16.22]{MeiseVogt}.

\subsection{Set Convergence}

In order to transfer results about convexity and star-shapedness of the $C$-numerical
range of matrices to Schatten-class operators, we need some basic facts about set convergence. 
We will use the Hausdorff metric on compact subsets (of $\mathbb C$) and the associated notion
of convergence, see, e.g., \cite{nadler1978}.\medskip

The distance between $z \in \mathbb C$ and any non-empty compact subset $A \subseteq \mathbb C$ is
defined by
\begin{align}\label{eq.Hausdorff-1}
d(z,A) := \min_{w \in A} d(z,w) = \min_{w \in A} |z-w|\,.
\end{align}
Based on \eqref{eq.Hausdorff-1}, the \emph{Hausdorff metric} $\Delta$ on the set of all non-empty
compact subsets of $\mathbb C$ is given by
\begin{align*}
\Delta(A,B) := \max\Big\lbrace \max_{z \in A}d(z,B),\max_{z \in B}d(z,A) \Big\rbrace.
\end{align*}
\noindent
The following result is proven in \cite[Lemma 2.5]{dirr_ve}.

\begin{lemma}\label{lemma_5}
Let $(A_n)_{n\in\mathbb N}$ and $(B_n)_{n\in\mathbb N}$ be bounded sequences of non-empty compact subsets
of $\mathbb C$ such that $\lim_{n\to\infty}A_n = A$, $\lim_{n\to\infty}B_n = B$ and let $(z_n)_{n\in\mathbb N}$
be any sequence of complex numbers with $\lim_{n\to\infty}z_n = z$. Then the following statements hold.
\begin{itemize}
\item[(a)] 
If $A_n\subseteq B_n$ for all $n\in\mathbb N$, then $A \subseteq B$.\vspace{4pt}
\item[(b)] 
The sequence $(\operatorname{conv}(A_n))_{n\in\mathbb N}$ of compact subsets converges to
$\operatorname{conv}(A)$, i.e.
\begin{align*}
\lim_{n\to\infty}\operatorname{conv}(A_n) = \operatorname{conv}(A)\,.
\end{align*}
\item[(c)] 
If $A_n$ is convex for all $n\in\mathbb N$, then $A$ is convex.\vspace{4pt}
\item[(d)] 
If $A_n$ is star-shaped with respect to $z_n$ for all $n\in\mathbb N$, then $A$ is star-shaped
with respect to $z$.
\end{itemize}
\end{lemma}

\section{Results}\label{sec:results}

Let $\mathcal H$ denote an arbitrary infinite-dimensional separable complex (i.s.c.) Hilbert space. Our goal will be to carry over the characterizations of the geometry of the $C$-numerical range $W_C (T)$, like star-shapedness or convexity, from the trace class \cite{dirr_ve} to conjugate Schatten-class operators on $\mathcal H$. 

\begin{definition}\label{defi_1}
Let $p,q\in [1,\infty]$ be conjugate. Then for $C\in\mathcal B^p(\mathcal H)$
and $T\in\mathcal B^q(\mathcal H)$, we define the \emph{$C$-numerical range} of $T$ to be
\begin{align*}
W_C (T):=\lbrace \operatorname{tr}(CU^\dagger TU)\,|\,U\in\mathcal B(\mathcal H)\text{ unitary}\rbrace\,.
\end{align*}
\end{definition}

\noindent
Note that the trace $\operatorname{tr}(CU^\dagger TU)$ is well-defined due to Lemma \ref{lemma_10}
and \ref{lemma_nu_hoelder}.

\medskip

Moreover, throughout this paper we need some mechanism to associate bounded operators on $\mathcal H$
with matrices. In doing so, let $(e_n)_{n\in\mathbb N} $ be some orthonormal basis of $\mathcal H$ and
let $(\hat e_i)_{i=1}^n$ be the standard basis of $\mathbb C^n$. For any $n\in\mathbb N$ we define 
\begin{align*}
\Gamma_n:\mathbb C^n\to \mathcal H,\qquad \hat{e_i}\mapsto \Gamma_n(\hat e_i):=e_i
\end{align*}
and its linear extension to all of $\mathbb C^n$. Next, let 
\begin{align}\label{cut_out_operator}
[\;\cdot\;]_n:\mathcal B(\mathcal H)\to\mathbb C^{n\times n},\qquad A\mapsto [A]_n:=\Gamma_n^\dagger A\Gamma_n
\end{align}
be the operator which ``cuts out'' the upper $n\times n$ block of (the matrix representation of) $A$ 
with respect to $(e_n)_{n\in\mathbb N} $.

\subsection{Star-Shapedness}

Our strategy is to transfer well-known properties of the finite-dimensional $[C]_n$-numerical range
of $[T]_n$ to $W_C(T)$ via the convergence results of Lemma \ref{lemma_5}.

\begin{lemma}\label{lemma_proj_strong_conv}
Let $p\in[1,\infty]$, $C \in \mathcal B^p(\mathcal H)$ and $(S_n)_{n\in\mathbb N}$ be a sequence in 
$\mathcal B(\mathcal H)$ which converges strongly to $S \in\mathcal B(\mathcal H)$. Then one has
$S_n C \to SC$, $CS_n^\dagger \to CS^\dagger$, and $S_nCS_n^\dagger \to SCS^\dagger$ for $n \to \infty$
with respect to the norm $\nu_p$.
\end{lemma}

\begin{proof}
The cases $p=1$ and $p=\infty$ are proven in \cite[Lemma 3.2]{dirr_ve}. As the proof for $p\in(1,\infty)$
is essentially the same, we sketch only the major differences. First, choose $K\in\mathbb N$ such that
\begin{align*}
\sum_{k=K+1}^\infty s_k(C)^p<\frac{\varepsilon^p}{(3\kappa)^p}\,,
\end{align*}
where $\kappa > 0$ satisfies $\|S\|\leq\kappa$ and $\|S_n\|\leq\kappa$ for all $n\in\mathbb N$. 
The existence of the constant $\kappa > 0$ is guaranteed by the uniform boundedness principle.
Then decompose $C=\sum_{k=1}^\infty s_k(C)\langle e_k,\cdot\rangle f_k$ into $C = C_1 + C_2$ with 
$C_1 := \sum_{k=1}^K s_k(C)\langle e_k,\cdot\rangle f_k$ finite-rank. By Lemma \ref{lemma_10}
one has
\begin{align*}
\nu_p(SC - S_nC)\leq \nu_p(SC_1-S_n C_1)+\Vert S\Vert\nu_p(C_2)+\Vert S_n\Vert\nu_p(C_2)
<\nu_p(SC_1-S_nC_1 )+\frac{2\varepsilon}{3}\,.
\end{align*}
Thus, what remains is to choose $N\in\mathbb N$ such that $\nu_p(SC_1-S_nC_1)<\varepsilon/3$
for all $n\geq N$. Starting from
\begin{align*}
\nu_p(SC_1-S_n C_1)
\leq  \sum_{k=1}^K s_k(C)\nu_p\big(\langle e_k,\cdot\rangle (Sf_k-S_nf_k)\big)=\sum_{k=1}^K s_k(C) \Vert Sf_k-S_nf_k \Vert\,,
\end{align*}
the strong convergence of $(S_n)_{n\in\mathbb N}$ yields $N \in \mathbb N$ such that
\begin{align*}
\Vert Sf_k - S_nf_k \Vert<\frac{\varepsilon}{3\sum_{k=1}^Ks_k(C)}
\end{align*}
for $k = 1, \dots, K$ and all $n\geq N$. This shows $\nu_p(SC - S_nC)\to 0$ as $n\to\infty$. All
other assertions are an immediate consequence of $\nu_p(A) = \nu_p(A^\dagger)$ for all 
$A \in \mathcal B^p(\mathcal H)$ and 
\begin{align*}
\nu_p(SCS^\dagger - S_nCS_n^\dagger) & \leq \Vert S\Vert\nu_p(CS^\dagger-CS_n^\dagger)+\nu_p(SC-S_nC)\Vert S_n\Vert \\
& \leq \kappa \big(\nu_p(CS^\dagger-CS_n^\dagger) + \nu_p(SC-S_nC)\big)\,.\qedhere
\end{align*}
\end{proof}

\begin{lemma}\label{strong_tr_conv}
Let $C\in\mathcal B^p(\mathcal H)$ and $T\in\mathcal B^q(\mathcal H)$ with $p,q\in [1,\infty]$ conjugate and let $(S_n)_{n\in\mathbb N}$ be a sequence in $\mathcal B(\mathcal H)$
which converges strongly to $S\in\mathcal B(\mathcal H)$. Then
\begin{align*}
\lim_{n\to\infty}\operatorname{tr}(CS_n^\dagger TS_n)=\operatorname{tr}(CS^\dagger TS)\,.
\end{align*}
Furthermore, the sequence of linear functionals $(\operatorname{tr}(CS_n^\dagger(\cdot)S_n))_{n\in\mathbb N}$ 
converges uniformly to $\operatorname{tr}(CS^\dagger (\cdot)S)$ on $\nu_q$-bounded subsets of 
$\mathcal B^q(\mathcal H)$, while the sequence $(\operatorname{tr}((\cdot)S_n^\dagger TS_n))_{n\in\mathbb N}$ 
converges uniformly to $\operatorname{tr}((\cdot)S^\dagger TS)$ on $\nu_p$-bounded subsets of 
$\mathcal B^p(\mathcal H)$.
\end{lemma}

\begin{proof}
The statement is a simple consequence of (\ref{eq:4}) and Lemma \ref{lemma_proj_strong_conv} as
\begin{align*}
|\operatorname{tr}(CS^\dagger TS) &-\operatorname{tr}(CS_n^\dagger TS_n)| = 
|\operatorname{tr}((SCS^\dagger-S_nCS_n^\dagger)T)|\\
&\leq \nu_p(SCS^\dagger-S_nCS_n^\dagger)\nu_q(T) \to 0
\quad\text{as } n\to\infty\,.\qedhere
\end{align*}
\end{proof}

\begin{theorem}\label{lemma_2}
Let $C\in\mathcal B^p(\mathcal H)$, $T\in\mathcal B^q(\mathcal H)$ with $p,q\in [1,\infty]$ conjugate be given. Furthermore, let $(e_n)_{n\in\mathbb N},(g_n)_{n\in\mathbb N}$ be arbitrary orthonormal bases of $\mathcal H$. Then
\begin{align*}
\lim_{n\to\infty}W_{[C]^e_{2n}}([T]^g_{2n})=\overline{W_C(T)}
\end{align*}
where $[\,\cdot\,]_k^e$ and $[\,\cdot\,]_k^g$ are the maps given by \eqref{cut_out_operator} with respect to $(e_n)_{n\in\mathbb N}$ and $(g_n)_{n\in\mathbb N}$, respectively.
\end{theorem}

\begin{proof}
The proof for $p=1$ and $q=\infty$ (or vice versa) given in \cite[Thm.~3.1]{dirr_ve} can be adjusted
to the case $p,q\in(1,\infty)$ by minimal modifications.
\end{proof}

Before proceeding with the star-shapedness of $\overline{W_C(T)}$, we need the following auxilliary
result to characterize the star-center later on.

\begin{lemma}\label{lemma_0_conv}
Let $C\in\mathcal B^p(\mathcal H)$ with $p\in( 1,\infty]$ and let $q\in[1,\infty)$ such that $p,q$ are conjugate. Furthermore, let $(e_n)_{n\in\mathbb N}$ be any orthonormal system in
$\mathcal H$. Then
\begin{align*}
\lim_{n\to\infty}\frac{1}{n^{1/q}}\sum_{k=1}^n\langle e_k,Ce_k\rangle=0\,.
\end{align*}
\end{lemma}

\begin{proof}
First, let $p=\infty$, so $q=1$. As $C$ is compact, by Lemma \ref{lemma_2b} (b), one has
$\lim_{k\to\infty}\langle e_k,Ce_k\rangle=0$ and thus the sequence of arithmetic means converges
to zero as well. Next, let $p\in(1,\infty)$ and $\varepsilon>0$. Moreover, we assume w.l.o.g.~$C\neq 0$
so $s_1(C) = \Vert C\Vert\neq 0$. As $C\in \mathcal B^p(\mathcal H)$, one can choose $N_1\in\mathbb N$
such that
\begin{align*}
\sum_{k=N_1+1}^\infty s_k(C)^p<\frac{\varepsilon^p}{2^p}
\end{align*}
and moreover $N_2\in\mathbb N$ such that
\begin{align*}
\frac{1}{n^{1/q}}<\frac{\varepsilon}{2\sum_{k=1}^{N_1}s_k(C)}
\end{align*}
for all $n\geq N_2$. Then, for any $n\geq N:=\max\lbrace N_1+1,N_2\rbrace$, by Lemma \ref{lemma_2b}
and H\"older's inequality we obtain
\begin{align*}
\Big|\frac{1}{n^{1/q}}\sum_{k=1}^n\langle e_k,Ce_k\rangle\Big|&\leq \frac{1}{n^{1/q}}\sum_{k=1}^{N_1}s_k(C)+\frac{1}{n^{1/q}}\sum_{k=N_1+1}^{n}s_k(C)\\
&\leq \frac{1}{n^{1/q}}\sum_{k=1}^{N_1}s_k(C)+\Big(\sum_{k=N_1+1}^{n} s_k(C)^p \Big)^{1/p}\Big(\sum_{k=N_1+1}^{n} \frac{1}{n} \Big)^{1/q}\\
&<\frac{\varepsilon}{2}+\Big(\sum_{k=N_1+1}^{\infty} s_k(C)^p \Big)^{1/p}\Big(\underbrace{ \frac{n-N_1}{n} }_{\leq 1}\Big)^{1/q}\leq\varepsilon\,.
\end{align*}
This concludes the proof.
\end{proof}

Now, our main result of this section reads as follows.

\begin{theorem}\label{theorem_1}
Let $C\in\mathcal B^p(\mathcal H)$ and $T\in\mathcal B^q(\mathcal H)$ with $p,q\in [1,\infty]$ conjugate. Then $\overline{W_C(T)}$ is star-shaped with respect to the origin.
\end{theorem}

\begin{proof}
Let $(e_n)_{n\in\mathbb N},(g_n)_{n\in\mathbb N}$ be arbitrary orthonormal bases of $\mathcal H$. For $n\in\mathbb N$, it is readily verified that
\begin{align*}
\frac{\operatorname{tr}([C]^e_{2n})\operatorname{tr}([T]^g_{2n})}{2n} &=\frac{\operatorname{tr}([C]^e_{2n})}{(2n)^{1/q}}\frac{\operatorname{tr}([T]^g_{2n})}{(2n)^{1/p}}\\
&=\Big( \frac{1}{(2n)^{1/q}}\sum_{j=1}^{2n} \langle e_j,Ce_j\rangle \Big)\Big( \frac{1}{(2n)^{1/p}}\sum_{j=1}^{2n} \langle g_j,Tg_j\rangle \Big)\,.
\end{align*}
Both factors converge and, by Lemma \ref{lemma_0_conv}, at least one of them goes to $0$ as $n\to\infty$. 
Moreover, $W_{[C]^e_{2n}}([T]^g_{2n})$ is star-shaped with respect to 
$(\operatorname{tr}([C]^e_{2n})\operatorname{tr}([T]^g_{2n})/(2n)$ for all $n\in\mathbb N$,
cf.~\cite[Thm.~4]{article_cheungtsing}. Thus Lemma \ref{lemma_5} (d) and Theorem \ref{lemma_2}
imply that $\overline{W_{C}(T)}$ is star-shaped with respect to $0 \in \mathbb{C}$, i.e. with respect
to the origin.
\end{proof}

\begin{remark}
The limit case $p=1$ and $q=\infty$ returns the known star-shapedness result in the case of trace-class
\cite[Thm.~3.3]{dirr_ve} because the essential numerical range satisfies $W_e(T)=\lbrace0\rbrace$ if (and only if) $T$ is compact \cite[Thm.~34.2]{bonsallduncan}.
\end{remark}

In analogy to the essential numerical range of a bounded linear operator as characterized in, e.g., \cite[Thm.~34.9]{bonsallduncan}, we introduce
the \emph{essential range} of a bounded linear functional $L\in(\mathcal B^q(\mathcal H))'$ via
\begin{align*}
W_e(L) :=  \Big\lbrace \lim_{n \to \infty} L(\langle f_n,\cdot \rangle f_n)\,\Big|\, (f_n)_{n \in \mathbb N}\;\text{ ONS of }\mathcal H\Big\rbrace \subset \mathbb C\,.
\end{align*}
By the canonical isomorphism $A \mapsto \operatorname{tr}(A\,\cdot\,)$ one has
$(\mathcal B^1(\mathcal H))' \simeq \mathcal B(\mathcal H)$ and 
$(\mathcal B^q(\mathcal H))' \simeq \mathcal B^p(\mathcal H)$ for $q \in (1,\infty]$ with
$p,q$ conjugate, refer to \cite[Thm.~V.15]{Schatten} and \cite[Prop.~16.26]{MeiseVogt}. Thus for $q \in [1,\infty]$, to each
$L\in(\mathcal B^q(\mathcal H))'$ we can associate a unique bounded linear operator $C \in \mathcal B(\mathcal H)$
if $q = 1$ and $C \in \mathcal B^p(\mathcal H)$ if $q \in (1,\infty]$, such that
\begin{equation}\label{eq:ess_range}
W_e(L) = W_e(C)\,.
\end{equation}
This shows that $W_e(L)$ is non-empty, compact and convex and, in particular, $W_e(L)= \{0\}$ 
for $q \in (1,\infty]$, cf.~\cite[Thm.~34.2]{bonsallduncan}. With the above terminology one has the following straightforward
conclusion.

\begin{corollary}
\begin{itemize}
\item[(a)]
Let $q\in (1,\infty]$ and $T\in\mathcal B^q(\mathcal H)$ be given. The closure of the image of the unitary
orbit of $T$ under any bounded linear functional $L\in(\mathcal B^q(\mathcal H))'$, i.e. the closure of
$$
L(\mathcal{O}_U(T)):=\lbrace L(U^\dagger TU)\,|\,U\in\mathcal B(\mathcal H)\text{ unitary}\rbrace\,,
$$
is star-shaped with 
respect to the origin.\vspace{4pt}
\item[(b)]
Let $q = 1$ and $T\in\mathcal B^1(\mathcal H)$ be given. The closure of the image of the unitary orbit
of $T$ under any bounded linear functional $L\in(\mathcal B^1(\mathcal H))'$ is star-shaped with 
respect to $\operatorname{tr}(T)W_e(L)$, i.e. all $z \in \operatorname{tr}(T)W_e(L)$ are possible
star centers.
\end{itemize}
\end{corollary}

\begin{proof}
(a) Let $q\in(1,\infty]$ with conjugate $p\in [1,\infty)$. Then, as seen above, 
$\mathcal B^p(\mathcal H) \simeq (\mathcal B^q(\mathcal H))'$ by means of the canonical map
$A\mapsto \operatorname{tr}(A\,\cdot\,)$. Now, $L(\mathcal{O}_U(T))=W_C(T)$
for some unique $C\in \mathcal B^p(\mathcal H)$ and thus, by Theorem \ref{theorem_1}, the closure of this
set is star-shaped with respect to $0 \in \mathbb C$.

\medskip
\noindent
(b) For $q=1$, again as seen above one has $(\mathcal B^1(\mathcal H))' \simeq \mathcal B(\mathcal H)$
and thus $L=\operatorname{tr}(B\,\cdot\,)$ for some $B\in\mathcal B(\mathcal H)$.
Hence, $L(\mathcal{O}_U(T))$ equals $W_T(B)$, cf.~\cite[Defi.~3.1]{dirr_ve}, and therefore is star-shaped with respect
to $\operatorname{tr}(T)W_e(B)=\operatorname{tr}(T)W_e(L)$, refer to \eqref{eq:ess_range} and \cite[Thm.~3.3]{dirr_ve}.
\end{proof}
\subsection{Convexity and the $C$-Spectrum}\label{sect_C_spectrum}
Convexity is definitely one of the most beautiful properties in the context of numerical ranges. A useful tool in order to characterize convexity of the $C$-numerical range is the $C$-spectrum, which was first introduced for matrices in \cite{article_marcus} and was generalized to infinite dimensions (more precisely, to trace-class operators) in \cite{dirr_ve}. Consequently, the next step is to transfer this concept and some of the known results to the Schatten-class setting.\medskip

In order to define the $C$-spectrum, we first have to fix the term \emph{eigenvalue sequence} of a 
compact operator $T \in \mathcal K(\mathcal H)$. In general, it is obtained by arranging the 
(necessarily countably many) non-zero eigenvalues in decreasing order with respect to their
absolute values and each eigenvalue is repeated as many times as its algebraic multiplicity\footnote{By
\cite[Prop.~15.12]{MeiseVogt}, every non-zero element $\lambda \in \sigma(T)$ of the spectrum
of $T$ is an eigenvalue of $T$ and has a well-defined finite algebraic multiplicity $\nu_a(\lambda)$,
e.g., $\nu_a(\lambda) := \dim \ker (T - \lambda I)^{n_0}$,  where $n_0 \in \mathbb N$ is the smallest 
natural number $n \in \mathbb N$ such that $\ker (T - \lambda I)^n = \ker (T - \lambda I)^{n+1}$.
\label{footnote_alg_mult}}. If only finitely many non-vanishing eigenvalues exist, the sequence is filled
up with zeros, see \cite[Ch.~15]{MeiseVogt}. For our purposes, we have to pass to a slightly 
\emph{modified eigenvalue sequence} as follows: 

\begin{itemize}
\item 
If the range of $T$ is infinite-dimensional and the kernel of $T$ is finite-dimensional,
then put $\operatorname{dim}(\operatorname{ker}T)$ zeros at the beginning of the eigenvalue
sequence of $T$. \vspace{4pt}
\item 
If the range and the kernel of $T$ are infinite-dimensional, mix infinitely many zeros into the
eigenvalue sequence of $T$.\footnote{Since in Definition \ref{defi_3} arbitrary permutations
will be applied to the modified eigenvalue sequence, we do not need to specify this mixing 
procedure further, cf. also \cite[Lemma 3.6]{dirr_ve}.}\vspace{4pt}
\item
If the range of $T$ is finite-dimensional leave the eigenvalue sequence of $T$ unchanged. 
\end{itemize}

Note that compact normal operators have a spectral decomposition of the form
\begin{align*}
T = \sum_{n=1}^\infty \tau_n \langle f_n, \cdot \rangle f_n
\end{align*}
where $(f_n)_{n \in \mathbb N}$ is an orthonormal basis of $\mathcal H$ and $(\tau_n)_{n \in \mathbb N}$
denotes the modified eigenvalue sequence of $T$, cf.~\cite[Thm.~VIII.§4.6]{berberian1976}. Hence it is
evident that for arbitrary $p\in[1,\infty)$, the absolute values of the non-vanishing eigenvalues and
the singular values of a \textit{normal} $T\in\mathcal B^p(\mathcal H)$ coincide and thus
\begin{align*}
\nu_p(T)=\Big(\sum_{n=1}^\infty |\tau_n|^p\Big)^{1/p}<\infty\,.
\end{align*}

\begin{definition}[$C$-spectrum]\label{defi_3}
Let $p,q\in[1,\infty]$ be conjugate. Then, for $C\in\mathcal B^p(\mathcal H)$ with modified
eigenvalue sequence $(\gamma_n)_{n\in\mathbb N}$ and $T\in\mathcal B^q(\mathcal H)$ with modified eigenvalue
sequence $(\tau_n)_{n\in\mathbb N}$, we define the $C$-\emph{spectrum} of $T$ to be
\begin{align*}
P_C(T):=\Big\lbrace \sum\nolimits_{n=1}^\infty \gamma_n\tau_{\sigma(n)} \,\Big|\, \sigma:\mathbb N \to\mathbb N \text{ is permutation}\Big\rbrace.
\end{align*}
\end{definition}

\noindent 
Due to H\"older's inequality and the standard estimate
$\sum_{n=1}^\infty  |\gamma_n(A)|^p \leq \sum_{n=1}^\infty s_n(A)^p$, cf.~\cite[Prop.~16.31]{MeiseVogt}, one has
$$
\sum_{n=1}^\infty |\gamma_n\tau_{\sigma(n)}|\leq \Big(\sum_{n=1}^\infty  s_n(C)^p \Big)^{1/p}\Big(\sum_{n=1}^\infty   s_n(T)^q\Big)^{1/q}=\nu_p(C)\nu_q(T)\,.
$$
Thus, the series $\sum\nolimits_{n=1}^\infty \gamma_n\tau_{\sigma(n)}$ in the definition of $P_C(T)$ are
well-defined and bounded by $\nu_p(C)\nu_q(T)$. \medskip

A comprehensive survey on basic results regarding the $C$-spectrum of a matrix can be found in
\cite[Ch.~6]{article_li_radii}. Below, in Theorem \ref{theorem_3}, we generalize some well-known
inclusion relations between the $C$-numerical range and the $C$-spectrum of matrices to Schatten-class
operators. Prior to this, however, we have to derive an approximation result similar to
Theorem \ref{lemma_2}.

\begin{theorem}\label{lemma_6}
Let $C\in\mathcal B^p(\mathcal H)$ and $T\in\mathcal B^q(\mathcal H)$ both be normal with 
$p,q\in [1,\infty]$ conjugate. Then
\begin{align*}
\lim_{n\to\infty}P_{[C]^e_n}([T]^g_n)= \overline{P_C(T)}\,.
\end{align*}
Here, $[\,\cdot\,]_k^e$ and $[\,\cdot\,]_k^g$ are the maps given by \eqref{cut_out_operator} with respect to the orthonormal bases $(e_n)_{n\in\mathbb N}$ and $(g_n)_{n\in\mathbb N}$ of $\mathcal H$ which diagonalize $C$ and
$T$, respectively.
\end{theorem}

\begin{proof}
A proof for $p=1,q=\infty$ (or vice versa) is given in \cite[Thm.~3.6]{dirr_ve} and can be adjusted 
to $p,q\in(1,\infty)$ by minimal modifications.
\end{proof}

Now our main result of this section reads as follows.

\begin{theorem}\label{theorem_3}
Let $C\in\mathcal B^p(\mathcal H)$ and $T\in\mathcal B^q(\mathcal H)$ with $p,q\in [1,\infty]$ 
conjugate. Then the following statements hold.
\begin{itemize}
\item[(a)] If either $C$ or $T$ is normal with collinear eigenvalues, then $\overline{W_C(T)}$ is convex.\vspace{4pt}
\item[(b)] If $C$ and $T$ both are normal, then
\begin{align*}
P_C(T)\subseteq W_C(T)\subseteq\operatorname{conv}(\overline{P_C(T)})\,.
\end{align*}
\item[(c)] If $C$ and $T$ both are normal and the eigenvalues of $C$ or $T$ are collinear, then
\begin{align*}
\overline{W_C(T)}=\operatorname{conv}(\overline{P_C(T)})\,.
\end{align*}
\end{itemize}
\end{theorem}

\begin{proof}
(a) W.l.o.g. let $C$ be normal with collinear eigenvalues. There exists an orthonormal basis
$(e_n)_{n\in\mathbb N}$ of $\mathcal H$ such that $C=\sum_{n=1}^\infty \gamma_n\langle e_n,\cdot\rangle e_n$.
Since $\gamma_n\to 0$ as $n\to\infty$, due to the collinearity assumption there exists
$\phi\in[0,2\pi)$ such that $e^{i\phi}C$ is hermitian. Thus, by Theorem \ref{lemma_2}, one has
\begin{align*}
\overline{W_C(T)}=\overline{W_{e^{i\phi}C}(e^{-i\phi}T)}=\lim_{n\to\infty} W_{[e^{i\phi}C]_{2n}^e}([e^{-i\phi}T]_{2n}^e)\,.
\end{align*}
As $[e^{i\phi}C]_{2n}^e\in\mathbb C^{2n\times 2n}$ is obviously hermitian for all $n\in\mathbb N$, it follows
that $W_{[e^{i\phi}C]_{2n}^e}([e^{-i\phi}T]_{2n}^e)$ is convex for $n\in\mathbb N$, cf.~\cite{article_poon}. 
Hence Lemma \ref{lemma_5} (c) yields the desired result.

\medskip
\noindent
(b) The statement can be proven completely analogously to \cite[Thm.~3.4 -- second inclusion]{dirr_ve}. 

\medskip
\noindent
(c) Finally, applying the closure and the convex hull to (b) yields
$\operatorname{conv}(\overline{P_C(T)})=\operatorname{conv}(\overline{W_C(T)})=\overline{W_C(T)}$,
where the last equality holds because of (a).
\end{proof}\bigskip

\textbf{Acknowledgements.} The authors are grateful to Thomas Schulte-Herbr\"uggen for valuable comments, and furthermore to the organizers of the WONRA 2018 which gave the inspiration for this follow-up paper. This work was supported in part by the \textit{Elitenetzwerk Bayern} through ExQM.
\bibliographystyle{tfnlm}
\bibliography{C_numerical_range_inf_dim_pq_2018_10_16}
\end{document}